\documentclass[11pt]{article}
\usepackage[intlimits]{amsmath}
\usepackage{latexsym,amsfonts,amsthm,amssymb, epsfig}
\usepackage{color,comment}
\usepackage[vcentering,dvips]{geometry}
\usepackage{pgfplots}
\usepackage{filecontents}
\usepackage{pgfplotstable}
\pgfplotsset{compat=1.14}
\usepackage{csvsimple}
\usepackage{booktabs}

\newtheorem{thm}{Theorem}
\newtheorem*{thm*}{Theorem}
\newtheorem{lem}[thm]{Lemma}

\theoremstyle{remark}
\newtheorem{rem}{Remark}

\theoremstyle{definition}
\newtheorem{dfn}[thm]{Definition}

\newcommand{\R}{\mathbb{R}}
\newcommand{\N}{\mathbb{N}}

\newcommand{\supp}{{\rm supp\ }}

\newcommand{\vol}{\operatornamewithlimits{vol}}

\title{On the volume of unit balls\\ of finite-dimensional Lorentz spaces}
\author{Anna Dole\v{z}alov\'a\footnote{Department of Mathematical Analysis, Faculty of Mathematics and Physics,
Charles University, Sokolovsk\'a 83, 186 00 Prague, Czech Republic;
the research of this author was supported by the grant P201/18/00580S of the Grant Agency of the Czech Republic},
Jan Vyb\'iral\footnote{Dept. of Mathematics FNSPE, Czech Technical University in Prague, Trojanova 13, 12000 Prague, Czech Republic;
the research of this author was supported by the grant P201/18/00580S of the Grant Agency of the Czech Republic
and from European Regional Development Fund-Project ``Center for Advanced Applied Science'' (No. CZ.02.1.01/0.0/0.0/16\_019/0000778),
\texttt{jan.vybiral@fjfi.cvut.cz}}}

\begin{document}
\maketitle
\begin{abstract}
We study the volume of unit balls $B^n_{p,q}$ of finite-dimensional Lorentz sequence spaces $\ell_{p,q}^n.$
We give an iterative formula for ${\rm vol}(B^n_{p,q})$ for the weak Lebesgue spaces with $q=\infty$ and explicit formulas for $q=1$ and $q=\infty.$
We derive asymptotic results for the $n$-th root of ${\rm vol}(B^n_{p,q})$ and show that
$[{\rm vol}(B^n_{p,q})]^{1/n}\approx n^{-1/p}$ for all $0<p<\infty$ and $0<q\le\infty.$
We study further the ratio between the volume of unit balls of weak Lebesgue spaces
and the volume of unit balls of classical Lebesgue spaces. We conclude with an application of
the volume estimates and characterize the decay of the entropy numbers of the embedding of
the weak Lebesgue space $\ell_{1,\infty}^n$ into $\ell_1^n.$
\end{abstract}

\section{Introduction and main results}

It was observed already in the first half of the last century (cf. the interpolation theorem of Marcinkiewicz \cite{Mar}) that the scale
of Lebesgue spaces $L_p(\Omega)$, defined on a subset $\Omega\subset \R^n$, is not sufficient to describe the fine properties
of functions and operators. After the pioneering work of Lorentz \cite{Lor1,Lor2}, Hunt defined in \cite{Hunt1, Hunt2}
a more general scale of function spaces $L_{p,q}(\Omega)$, the so-called Lorentz spaces. This scale includes Lebesgue spaces
as a special case (for $p=q$) and Lorentz spaces have found applications in many areas of mathematics, including harmonic analysis (cf. \cite{Graf1,Graf2})
and the analysis of PDE's (cf. \cite{LR,Meyer}).

If $\Omega$ is an atomic measure space (with all atoms of the same measure), one arrives naturally at the definition of Lorentz spaces
on (finite or infinite) sequences. If $n$ is a positive integer and $0<p\le\infty$, then the Lebesgue $n$-dimensional space $\ell_p^n$ is
$\R^n$ equipped with the (quasi-)norm
\begin{equation*}
\|x\|_p=\begin{cases}\displaystyle \Bigl(\sum_{j=1}^n |x_j|^p\Bigr)^{1/p},&\quad\text{for}\ 0<p<\infty,\\
\displaystyle\max_{j=1,\dots,n}|x_j|,&\quad\text{for}\ p=\infty
\end{cases}
\end{equation*}
for every $x=(x_1,\dots,x_n)\in\R^n$. We denote by $B_p^n$ its unit ball
\begin{equation}\label{eq:defBp}
B_p^n=\{x\in\R^n:\|x\|_p\le 1\}.
\end{equation}
If $0<p,q\le\infty$, then the Lorentz space $\ell_{p,q}^n$ stands for $\R^n$ equipped with the \mbox{(quasi-)norm}
\begin{equation}\label{eq:defpq}
\|x\|_{p,q}=\|k^{\frac{1}{p}-\frac{1}{q}}x_k^*\|_q,
\end{equation}
where $x^*=(x_1^*,\dots,x_n^*)$ is the non-increasing rearrangement of $(|x_1|,\dots,|x_n|)$.
If $p=q$, then $\ell_{p,p}^n=\ell_p^n$ are again the Lebesgue sequence spaces.
If $q=\infty$, then the space $\ell_{p,\infty}^n$ is usually referred to as a weak Lebesgue space.

Similarly to \eqref{eq:defBp}, we denote by $B_{p,q}^n$ the unit ball of $\ell_{p,q}^n$, i.e. the set
\begin{equation}
B_{p,q}^n=\{x\in\R^n:\|x\|_{p,q}\le 1\}.
\end{equation}
Furthermore, $B^{n,+}_{p}$ (or $B_{p,q}^{n,+}$) 
will be the set of vectors from $B_p^n$ (or $B_{p,q}^{n}$) with all coordinates non-negative.

Lorentz spaces of (finite or infinite) sequences have been used extensively in different areas of mathematics.
They form a basis for many operator ideals of Pietsch, cf. \cite{Pietsch1, Pietsch2, Triebel2},
they play an important role in the interpolation theory, cf. \cite{BS,BL,Haroske,LiPe}, and their weighted counterparts are the main building blocks
of approximation function spaces, cf. \cite{ST,Triebel1}.
Weak Lebesgue sequence spaces (i.e. Lorentz spaces with $q=\infty$) were used by Cohen, Dahmen, Daubechies, and DeVore \cite{CDDD}
to characterize functions of bounded variation. Lorentz spaces further appear in approximation theory \cite{DeVore,DeLo,DePeTe} and signal processing \cite{IEEE1,IEEE2,FR}.

The volume of unit balls of classical Lebesgue sequence spaces $B_p^n$ is known since the times of Dirichlet \cite{Dirichlet},
who showed for $0<p\le\infty$ that
\begin{equation}\label{eq:volBp}
\vol(B_p^n)=2^n\cdot\frac{\Gamma\bigl(1+\frac{1}{p}\bigr)^n}{\Gamma\bigl(1+\frac{n}{p}\bigr)}.
\end{equation}
Here, $\vol(A)$ stands for the Lebesgue measure of a (measurable) set $A\subset\R^n$
and $\Gamma(x)=\int_0^\infty t^{x-1}e^{-t}dt$ is the Gamma function for $x>0$.
Since then, \eqref{eq:volBp} and its consequences
play an important role in many results about finite-dimensional Lebesgue spaces, cf. \cite{Pisier}.
Although many properties of Lorentz sequence spaces  
were studied in detail earlier (cf. \cite{LorTheo2, LorTheo1, BS,BL, LorTheo5, LorTheo3, LorTheo4}), there seems to be only very little known about
the volume of their unit balls. The aim of this work is to fill to some extent this gap.

We present two ways, which lead to recursive formulas for $\vol(B_{p,\infty}^n)$ if $0<p<\infty$.
The first one (cf. Theorem \ref{thm:ind:1})
\begin{equation}\label{eq:intro:1}
\vol(B^{n,+}_{p,\infty})=\sum_{j=1}^n (-1)^{j-1}{n\choose j}n^{-j/p}\vol(B^{n-j,+}_{p,\infty})
\end{equation}
is quite well suited for calculating $\vol(B_{p,\infty}^n)$ for moderate values of $n$
and we present also some numerical results on the behavior of this quantity for different values of $p$.
Although an explicit formula for $\vol(B_{p,\infty}^n)$ can be derived from \eqref{eq:intro:1}, cf. Theorem \ref{thm:ind:2},
due to its combinatorial nature it seems to be only of a limited practical use. In Section \ref{sec:integral}
we derive the same formula with the help of iterated multivariate integrals, very much in the spirit of the original proof of Dirichlet.

Surprisingly, a simple explicit formula can be given for $\vol(B_{p,1}^n)$ for the full range of $0<p\le\infty$. Indeed, we show in Theorem \ref{thm:q1} that
$$
\vol(B_{p,1}^n)=2^n\prod_{k=1}^n\frac{1}{\varkappa_p(k)},\quad \text{where}\quad \varkappa_p(k)=\sum_{j=1}^kj^{1/p-1}.
$$
If $p=1$, then $\varkappa_1(k)=k$ and this formula reduces immediately to the well-known relation $\vol(B_{1}^n)=2^n/n!.$

Using Stirling's formula, \eqref{eq:volBp} implies that $[\vol(B_p^n)]^{1/n}\approx n^{-1/p}$ for all $0<p<\infty$ with
the constants of equivalence independent of $n$. Using the technique of entropy numbers, we show in Theorem \ref{thm:asym:1} that essentially the same is true
for the whole scale of Lorentz spaces $\ell_{p,q}^n$ (with a remarkable exception for $p=\infty$, cf. Theorem \ref{thm:asym:inf1}).

It is a very well known fact (cf. Theorem \ref{thm:emb:1}) that $B_{p}^n\subset B_{p,\infty}^n$
for all $0<p\le\infty$ and it is a common folklore to consider the unit balls of weak Lebesgue spaces (i.e. Lorentz spaces with $q=\infty$)
as the ``slightly larger'' counterparts of the unit balls of Lebesgue spaces with the same summability parameter $p$.
This intuition seems to be further confirmed by Theorem \ref{thm:asym:1}, which shows that the quantities
$[\vol(B^n_p)]^{1/n}$ and $[\vol(B^n_{p,\infty})]^{1/n}$ are equivalent to each other with constants independent on $n$.
On the other hand, we show in Theorem \ref{thm:ratio} that $\vol(B^n_{p,\infty})/\vol(B^n_{p})$ grows exponentially
in $n$ at least for $p\le 2$. We conjecture (but it remains an open problem)
that the same is true for all $p<\infty.$

We conclude our work by considering the entropy numbers of the embeddings between Lorentz spaces
of finite dimension, which complements the seminal work of Edmunds and Netrusov \cite{EN}.
We characterize in Theorem \ref{thm:entropy}
the decay of the entropy numbers
$e_k(id:\ell_{1,\infty}^n\to \ell_{1}^n)$, which turns out to exhibit a rather
unusual behavior, namely
$$
e_k(id:\ell_{1,\infty}^n\to \ell_1^n)\approx\begin{cases} \log(1+n/k),\quad 1\le k\le n,\\
2^{-\frac{k-1}{n}},\quad k\ge n
\end{cases}
$$
with constants of equivalence independent of $k$ and $n$.
We see that after a logarithmic decay for $1\le k \le n$, the exponential decay in $k$
takes over for $k\ge n.$


\section{Recursive and explicit formulas}

In this section, we present different formulas for the volume of unit balls of Lorentz spaces
for two special cases, namely for the weak Lebesgue spaces with $q=\infty$ and for Lorentz spaces with $q=1.$
Surprisingly, different techniques have to be used in these two cases.

\subsection{Weak Lebesgue spaces}

We start by the study of weak Lebesgue spaces, i.e. the Lorentz spaces $\ell_{p,\infty}^n$. If $p=\infty$, then $\ell_{p,\infty}^n=\ell_\infty^n$.
Therefore, we restrict ourselves to $0<p<\infty$ in this section.

\subsubsection{Using the inclusion-exclusion principle}
In this section, we assume the convention
$$
\vol(B^{1,+}_{p,\infty})=\vol(B^{0,+}_{p,\infty})=1
$$
for every $0<p<\infty.$
\begin{thm}\label{thm:ind:1} Let $n\in\N$ and $0<p<\infty$. Then
\begin{equation}\label{eq:ind:1}
\vol(B^{n,+}_{p,\infty})=\sum_{j=1}^n (-1)^{j-1}{n\choose j}n^{-j/p}\vol(B^{n-j,+}_{p,\infty}).
\end{equation}
\end{thm}
\begin{proof} For $1\le k \le n$, we denote $A_k=\{x\in B^{n,+}_{p,\infty}:x_k\le n^{-1/p}\}$.
If $x\in B^{n,+}_{p,\infty}$, then at least one of the coordinates of $x$ must be smaller than or equal to $n^{-1/p}$. Therefore
$$
B^{n,+}_{p,\infty}=\bigcup_{j=1}^n A_j.
$$
For a non-empty index set $K\subset \{1,\dots,n\}$, we denote
$$
A_K=\bigcap_{k\in K}A_k=\{x\in B_{p,\infty}^{n,+}:x_k\le n^{-1/p}\ \text{for all}\ k\in K\}.
$$
If we denote by $x_{K^c}$ the restriction of $x$ onto $K^c=\{1,\dots,n\}\setminus K$, then the $j^{\rm th}$
largest coordinate of $x_{K^c}$ can be at most $j^{-1/p}$, i.e. $x_{K^c}\in B^{n-|K|,+}_{p,\infty}$.
Here, $|K|$ stands for the number of elements in $K$.
We therefore obtain
\begin{equation}\label{eq:volK}
\vol(A_K)=\Bigl(\prod_{k\in K}n^{-1/p}\Bigr)\cdot \vol(B^{n-|K|,+}_{p,\infty})=n^{-|K|/p}\cdot \vol(B^{n-|K|,+}_{p,\infty}).
\end{equation}

Finally, we insert \eqref{eq:volK} into the inclusion-exclusion principle and obtain
\begin{align*}
\vol(B^{n,+}_{p,\infty})&=\sum_{\emptyset\not=K\subset\{1,\dots,n\}}(-1)^{|K|-1}\vol(A_K)\\
&=\sum_{\emptyset\not=K\subset\{1,\dots,n\}}(-1)^{|K|-1}n^{-|K|/p}\vol(B^{n-|K|,+}_{p,\infty})\\
&=\sum_{j=1}^n (-1)^{j-1}{n\choose j}n^{-j/p}\vol(B^{n-j,+}_{p,\infty}).
\end{align*}
\end{proof}
The relation \eqref{eq:ind:1} is already suitable for calculation of $\vol(B^{n}_{p,\infty})$ for moderate values of $n$, cf. Table \ref{table_infty}.
Let us remark, that $\vol(B^n_{1,\infty})$ is maximal for $n=4$ and $\vol(B^n_{2,\infty})$ attains its maximum at $n=18.$

\begin{table}[h t p]\centering
\pgfplotstableread[col sep = semicolon]{q1prehled15.txt}\loadedtable
\setlength{\tabcolsep}{8pt}
\pgfplotstabletypeset[
                precision=3,
                every head row/.style={
                        before row=\toprule,after row=\midrule},
                        every last row/.style={
                        after row=\bottomrule},
                columns/n/.style={int detect,
                        column name={$n$}},
                columns/pul/.style={sci, sci zerofill, sci sep align,
                        column name={$p=1/2$}},
                columns/jedna/.style={sci, sci zerofill, sci sep align,
                        column name={$p=1$}},
                columns/dva/.style={sci, sci zerofill, sci sep align,
                        column name={$p=2$}},
                columns/sto/.style={sci, sci zerofill, sci sep align,
                        column name={$p=100$}},
                        ]\loadedtable
\caption{$\text{vol}(B^n_{p,\infty})$ for dimensions up to 15}
\label{table_infty}
\end{table}

Next we exploit Theorem \ref{thm:ind:1} to give a certain explicit result about the volume of unit balls of weak Lebesgue spaces.
For this, we denote by ${\bf K}_n$ the set of integer vectors of finite length $k=(k_1,\dots,k_j)$ with positive
coordinates $k_1,\dots,k_j$, which sum up to $n$. We denote by $\ell(k)=j$ the length of $k\in{\bf K}_n.$
Similarly, we denote by ${\bf M}_n$ the set of all increasing sequences $m=(m_0,\dots,m_j)$
which grow from zero to $n$, i.e. with $0=m_0<m_1<\dots<m_j=n.$ The quantity $\ell(m)=j$ is again the length of $m\in{\bf M}_n$.
Hence,
\begin{align*}
{\bf K}_n&:=\{k=(k_1,\dots,k_j):k_i\in\N, \sum_{i=1}^j k_i=n\},\\
{\bf M}_n&:=\{m=(m_0,\dots,m_j):m_i\in\N_0,\ 0=m_0<m_1<\dots<m_j=n\}.
\end{align*}
For $k\in{\bf K}_n$, we also write
$$
{n\choose k}={n\choose k_1,\dots,k_{\ell(k)}}=\frac{n!}{k_1!\dots k_{\ell(k)}!}
$$
The explicit formula for $\vol(B_{p,\infty}^n)=2^n\vol(B_{p,\infty}^{n,+})$ is then presented in the following theorem.
\begin{thm}\label{thm:ind:2} Let $0<p<\infty$ and $n\in\N.$ Then
\begin{align}
\notag \vol(B^{n,+}_{p,\infty})&=\sum_{k\in {\bf K}_n} (-1)^{n+\ell(k)}{n\choose k}
\prod_{l=1}^{\ell(k)}\Bigl(n-\sum_{i=1}^{l-1}k_i\Bigr)^{-k_l/p}\\
\label{eq:ind:2}&=n!\sum_{m\in {\bf M}_n} (-1)^{n+\ell(m)}
\prod_{l=0}^{\ell(m)-1} \frac{(n-m_l)^{-(m_{l+1}-m_{l})/p}}{(m_{l+1}-m_{l})!}.
\end{align}
\end{thm}
\begin{proof}
First, we prove the second identity in \eqref{eq:ind:2}. Indeed, the mapping $k=(k_1,\dots,k_j)\to (0,k_1,k_1+k_2,\dots,\sum_{i=1}^jk_i)$
maps one-to-one ${\bf K}_n$ onto ${\bf M}_n$, preserving also the length of the vectors.

Next, we proceed by induction to show the first identity of \eqref{eq:ind:2}. For that sake, we denote ${\bf K}_0=\{0\}$ with $\ell(0)=0$.
With this convention, \eqref{eq:ind:2}
is true for both $n=0$ and $n=1$, where both the sides of \eqref{eq:ind:2} are equal to one.
The rest follows from \eqref{eq:ind:1}.
Indeed, we obtain
\begin{align*}
\vol(B^{n,+}_{p,\infty})&=\sum_{j=1}^n (-1)^{j-1}{n\choose j}n^{-j/p}\vol(B^{n-j,+}_{p,\infty})\\
&=\sum_{j=1}^n (-1)^{j-1}{n\choose j}n^{-j/p}\sum_{k\in{\bf K}_{n-j}}(-1)^{n-j+\ell(k)}{n-j\choose k}\prod_{l=1}^{\ell(k)}\Bigl(n-j-\sum_{i=1}^{l-1}k_i\Bigr)^{-k_l/p}\\
&=\sum_{j=1}^n \sum_{k\in{\bf K}_{n-j}} (-1)^{n+\ell(k)-1}{n\choose j}{n-j\choose k}n^{-j/p}\prod_{l=1}^{\ell(k)}\Bigl(n-j-\sum_{i=1}^{l-1}k_i\Bigr)^{-k_l/p}\\
&=\sum_{\nu\in{\bf K}_{n}} (-1)^{n+\ell(\nu)}{n\choose \nu}n^{-\nu_1/p}\prod_{l=1}^{\ell(\nu)-1}\Bigl(n-\sum_{i=1}^{l}\nu_i\Bigr)^{-\nu_{l+1}/p}\\
&=\sum_{\nu\in{\bf K}_{n}} (-1)^{n+\ell(\nu)}{n\choose \nu}\prod_{l=1}^{\ell(\nu)}
\Bigl(n-\sum_{i=1}^{l-1}\nu_i\Bigr)^{-\nu_{l}/p},
\end{align*}
where we identified the pair $(j,k)$ with $1\le j\le n$ and $k=(k_1,\dots,k_{\ell(k)})\in {\bf K}_{n-j}$  with $\nu=(j,k_1,\dots,k_{\ell(k)})\in{\bf K}_n$.
If $j=n$, then the pair $(n,0)$ is identified with $\nu=(n)$. In any case, $\ell(\nu)=\ell(k)+1.$
\end{proof}

\subsubsection{Integral approach}\label{sec:integral}

The result of Theorem \ref{thm:ind:2} can be obtained also by an iterative evaluation of integrals, resembling
the approach of the original work of Dirichlet \cite{Dirichlet}.
To begin with, we define a scale of expressions which for some specific choice of parameters lead to a formula for $\vol(B^n_{p,\infty})$.
\begin{dfn}\label{def:V} Let $m\in\N_0$ and $n\in\N$. Let $a\in\R^n$ be a decreasing positive vector, i.e. $a=(a_1,\dots,a_n)$
with $a_1>a_2>\dots>a_n>0.$ We denote
\begin{equation}\label{eq:def:V}
V^{(m)}(n,a)=\int_0^{a_n}\int_{x_n}^{a_{n-1}}\dots\int_{x_2}^{a_1}x_1^mdx_1\dots dx_{n-1}dx_n.
\end{equation}
\end{dfn}
The domain of integration in \eqref{eq:def:V} is defined by the following set of conditions
\begin{align*}
0\le x_n\le a_n,\quad x_n\le x_{n-1}\le a_{n-1},\quad \dots,\quad x_2\le x_1\le a_1,
\end{align*}
which can be reformulated also as
\begin{align*}
0\le x_n\le x_{n-1}\le \dots\le x_2\le x_1\quad\text{and}\quad x_j\le a_j\ \text{for all}\ j=1,\dots,n.
\end{align*}
Hence, the integration in \eqref{eq:def:V} goes over the cone of non-negative non-increasing vectors in $\R^n$ intersected with the set
$\{x\in\R^n:x^*_j\le a_j\ \text{for}\ j=1,\dots,n\}.$ If we set $a^{(p)}=(a_1^{(p)},\dots,a_n^{(p)})$ with $a_j^{(p)}=j^{-1/p}$ for $0<p<\infty$ and $1\le j\le n$,
this set coincides with $B^{n,+}_{p,\infty}$. Finally, considering all the possible reorderings of $x$, we get
\begin{equation}\label{eq:BV1}
\vol(B^{n,+}_{p,\infty})=n!\cdot V^{(0)}(n,a^{(p)}).
\end{equation}
In what follows, we simplify the notation by assuming $V^{(0)}(0,\emptyset)=1.$
The integration in \eqref{eq:def:V} leads to the following recursive formula for $V^{(m)}(n,a)$.
\begin{lem}\label{lem:ind:V1} Let $m\in\N_0$, $n\in\N$ and $a\in\R^n$ with $a_1>a_2>\dots>a_n>0.$ Then
\begin{equation}\label{eq:ind:V1}
V^{(m)}(n,a)=\sum_{i=1}^{n}(-1)^{i+1}\frac{a_i^{m+i}m!}{(m+i)!}V^{(0)}(n-i,(a_{i+1},\dots,a_n)).
\end{equation}
\end{lem}
\begin{proof}
First, we obtain
\begin{align*}
V^{(m)}(n,a)&=\int_0^{a_n}\int_{x_n}^{a_{n-1}}\dots\int_{x_2}^{a_1}x_1^mdx_1\dots dx_{n-1}dx_n\\
&=\frac{1}{m+1}\int_0^{a_n}\int_{x_n}^{a_{n-1}}\dots\int_{x_3}^{a_2}(a_1^{m+1}-x_2^{m+1})dx_2\dots dx_{n-1}dx_n\\
&=\frac{a_1^{m+1}}{m+1}V^{(0)}(n-1,(a_2,\dots,a_n))-\frac{1}{m+1}V^{(m+1)}(n-1,(a_2,\dots,a_n)).
\end{align*}
The proof of \eqref{eq:ind:V1} now follows by induction over $n$. For $n=1$, $V^{(m)}(1,(a_1))=\frac{a_1^{m+1}}{m+1}$,
which is in agreement with \eqref{eq:ind:V1}. To simplify the notation later on,
we write $a_{[k,l]}=(a_k,\dots,a_l)$ for every $1\le k\le l\le n$.
We assume, that \eqref{eq:ind:V1} holds for $n-1$ and calculate
\begin{align*}
V^{(m)}(n,a)&=\frac{a_1^{m+1}}{m+1}V^{(0)}(n-1,a_{[2,n]})-\frac{1}{m+1}V^{(m+1)}(n-1,(a_{[2,n]}))\\
&=\frac{a_1^{m+1}}{m+1}V^{(0)}(n-1,a_{[2,n]})-
\frac{1}{m+1}\sum_{i=1}^{n-1}(-1)^{i+1}\frac{a_{i+1}^{m+1+i}(m+1)!}{(m+1+i)!}V^{(0)}(n-1-i,(a_{[i+2,n]}))\\
&=\frac{a_1^{m+1}}{m+1}V^{(0)}(n-1,a_{[2,n]})+
\sum_{j=2}^{n}(-1)^{j+1}\frac{a_{j}^{m+j}m!}{(m+j)!}V^{(0)}(n-j,(a_{[j+1,n]}))\\
&=\sum_{j=1}^{n}(-1)^{j+1}\frac{a_{j}^{m+j}m!}{(m+j)!}V^{(0)}(n-j,(a_{[j+1,n]})),
\end{align*}
which finishes the proof of \eqref{eq:ind:V1}.
\end{proof}
Lemma \ref{lem:ind:V1} allows for a different proof of Theorem \ref{thm:ind:2}.
\begin{proof}[Alternative proof of Theorem \ref{thm:ind:2}]
By \eqref{eq:BV1}, we need to calculate $V^{(0)}(n,a)$ and then substitute $a=a^{(p)}.$
We show by induction that
\begin{equation}\label{eq:BV2}
V^{(0)}(n,a)=\sum_{m\in{\bf M}_n}(-1)^{n+\ell(m)}\prod_{l=0}^{\ell(m)-1}\frac{a_{n-m_l}^{m_{l+1}-m_l}}{(m_{l+1}-m_l)!}.
\end{equation}
If $n=1$, then both sides of \eqref{eq:BV2} are equal to $a_1.$ For general $n$, we obtain by \eqref{eq:ind:V1}
\begin{align*}
V^{(0)}(n,a)&=\sum_{i=1}^{n}(-1)^{i+1}\frac{a_i^{i}}{i!}V^{(0)}(n-i,(a_{i+1},\dots,a_n))\\
&=\sum_{i=1}^{n}(-1)^{i+1}\frac{a_i^{i}}{i!} \sum_{m\in{\bf M}_{n-i}}(-1)^{n-i+\ell(m)}\prod_{l=0}^{\ell(m)-1}\frac{a_{n-i-m_l+i}^{m_{l+1}-m_l}}{(m_{l+1}-m_l)!}\\
&=\sum_{i=1}^{n}\sum_{m\in{\bf M}_{n-i}}(-1)^{n+1+\ell(m)}\frac{a_i^{i}}{i!} \prod_{l=0}^{\ell(m)-1}\frac{a_{n-m_l}^{m_{l+1}-m_l}}{(m_{l+1}-m_l)!}\\
&=\sum_{\mu\in{\bf M}_n} (-1)^{n+\ell(\mu)} \frac{a_{n-\mu_{\ell(\mu)-1}}^{\mu_{\ell(\mu)}-\mu_{\ell(\mu)-1}}}{(\mu_{\ell(\mu)}-\mu_{\ell(\mu)-1})!}
\prod_{l=0}^{\ell(\mu)-2}\frac{a_{n-\mu_l}^{\mu_{l+1}-\mu_l}}{(\mu_{l+1}-\mu_l)!}\\
&=\sum_{\mu\in{\bf M}_n} (-1)^{n+\ell(\mu)} \prod_{l=0}^{\ell(\mu)-1}\frac{a_{n-\mu_l}^{\mu_{l+1}-\mu_l}}{(\mu_{l+1}-\mu_l)!},
\end{align*}
where we identified the pair $(i,m)$ with $1\le i \le n$ and $m=(m_0,\dots,m_{\ell(m)})\in{\bf M}_{n-i}$ with
$\mu=(\mu_0,\dots,\mu_{\ell(\mu)})=(m_0,\dots,m_{\ell(m)},n)\in{\bf M}_n$. Hence, $\ell(\mu)=\ell(m)+1.$
\end{proof}

\subsection{Lorentz spaces with $q=1$}
We give an explicit formula for $\vol(B_{p,1}^n)$, which takes a surprisingly simple form.
The approach is based on the polarization. Recall, that for $0<p\le \infty$, the \mbox{(quasi-)norm}
$\|x\|_{p,1}$ is defined as
$$
\|x\|_{p,1}=\sum_{k=1}^n k^{1/p-1}x_k^*.
$$
\begin{thm}\label{thm:q1} Let $0<p\le \infty$. Then
$$
\vol(B_{p,1}^n)=2^n\prod_{k=1}^n\frac{1}{\varkappa_p(k)},\quad \text{where}\quad \varkappa_p(k)=\sum_{j=1}^kj^{1/p-1}.
$$
\end{thm}
\begin{proof}
Let $f:[0,\infty)\to \R$ be a smooth non-negative function with a sufficient decay at infinity (later on, we will just choose $f(t)=e^{-t}$).
Then
\begin{align}
\notag\int_{\R^n}f(\|x\|_{p,1})dx&=-\int_{\R^n}\int_{\|x\|_{p,1}}^\infty f'(t)dtdx=-\int_0^\infty f'(t)\int_{x:\|x\|_{p,1}\le t}1dx dt\\
\label{eq:q1:1}&=-\int_0^\infty f'(t)\vol(\{x:\|x\|_{p,1}\le t\})dt\\
\notag&=-\int_0^\infty t^n f'(t)\vol(\{x:\|x\|_{p,1}\le 1\})dt=-\vol(B^n_{p,1})\int_0^\infty t^n f'(t)dt.
\end{align}
For the choice $f(t)=e^{-t}$, we get
\begin{equation}\label{eq:q1:2}
-\int_0^\infty t^n f'(t)dt=\int_0^\infty t^ne^{-t}dt=\Gamma(n+1)=n!.
\end{equation}
It remains to evaluate
\begin{align*}
I^n_{p}=\int_{\R^n}\exp(-\|x\|_{p,1})dx=\int_{\R^n}\exp\Bigl(-\sum_{k=1}^n k^{1/p-1}x_k^*\Bigr)dx
=2^n\cdot n!\cdot\int_{{\mathcal C}_+^n}\exp\Bigl(-\sum_{k=1}^n k^{1/p-1}x_k\Bigr)dx,
\end{align*}
where
$$
{\mathcal C}_+^n=\Bigl\{x\in\R^n:x_1\ge x_2\ge\dots\ge x_n\ge 0\Bigr\}.
$$
We denote for $t\ge 0$, $0<p\le\infty$, and $n\in \N$
$$
A(n,p,t)=\int_{{\mathcal C}_{t,+}^n}\exp\Bigl(-\sum_{k=1}^n k^{1/p-1}x_k\Bigr)dx,
$$
where ${\mathcal C}_{t,+}^n=\Bigl\{x\in\R^n:x_1\ge x_2\ge\dots\ge x_n\ge t\Bigr\}$, i.e. $I_p^n=2^n\cdot n!\cdot A(n,p,0).$
We observe that
\begin{equation}\label{eq:int:1}
A(1,p,t)=\int_t^\infty e^{-u}du=e^{-t}
\end{equation}
and
\begin{align}
\notag A(n,p,t)&=\int_{t}^\infty \exp\Bigl({-n^{1/p-1}x_n}\Bigr)\int_{x_n}^\infty \exp\Bigl({-(n-1)^{1/p-1}x_{n-1}}\Bigr)\dots \int_{x_2}^\infty \exp({-x_1})dx_1\dots dx_{n-1}dx_n\\
\label{eq:int:2}
&=\int_{t}^\infty \exp\Bigl({-n^{1/p-1}x_n}\Bigr)A(n-1,p,x_n)dx_n.
\end{align}
Combining \eqref{eq:int:1} and \eqref{eq:int:2}, we prove by induction
\begin{align*}
A(n,p,t)=\prod_{k=1}^n\frac{1}{\varkappa_p(k)}\exp(-\varkappa_p(n) t),\quad \text{where}\quad \varkappa_p(k)=\sum_{j=1}^kj^{1/p-1}
\end{align*}
and
\begin{equation}\label{eq:q1:3}
I_p^n=2^n\cdot n!\cdot \prod_{k=1}^n\frac{1}{\varkappa_p(k)}.
\end{equation}
Finally, we combine \eqref{eq:q1:1} with \eqref{eq:q1:2} and \eqref{eq:q1:3} and obtain
$$
\vol(B_{p,1}^n)=2^n\prod_{k=1}^n\frac{1}{\varkappa_p(k)}.
$$
\end{proof}

\begin{rem}
Let us point out that for $p=1$, we get $\varkappa_1(k)=k$ and we recover the very well known formula $\vol(B^n_{1,1})=2^n/n!$.
The application of the polarization identity to other values of $q\not=1$ is also possible, but one arrives to an $n$-dimensional integral,
which (in contrary to $I_p^n$) seems to be hard to compute explicitly.
\end{rem}

\section{Asymptotic behavior}

Volumes of convex and non-convex bodies play an important role in many areas of mathematics, cf. \cite{Pisier}.
Nevertheless, for most of the applications we do not need the exact information about the volume, it is often enough
to apply good lower and/or upper bounds of this quantity. For example, for the use in local Banach theory,
it is sometimes sufficient to have some asymptotic bounds on $\vol(B_p^n)$ for $n$ large.
In this section, we provide two such estimates.

\subsection{Asymptotic behavior of $\vol(B_{p,q}^n)^{1/n}$}

The first quantity, we would like to study, is the $n$-th root of $\vol(B^n_{p,q})$. In the Lebesgue case $q=p$, \eqref{eq:volBp}
can be combined with the Stirling's formula (cf. \cite{WiW})
\begin{equation}\label{eq:Stirling}
\Gamma(t)=(2\pi)^{1/2}t^{t-1/2}e^{-t}e^{\theta(t)/t},\quad 0<t<\infty,
\end{equation}
where $0<\theta(t)<1/12$ for all $t>0$, to show that
\begin{equation}\label{eq:asymp:p}
\vol(B_p^n)^{1/n}\approx n^{-1/p},
\end{equation}
where the constants of equivalence do not depend on $n$. Combining \eqref{eq:asymp:p} with the embedding (cf. Theorem \ref{thm:emb:1})
$$
B_p^n\subset B_{p,\infty}^n\subset (1+\log(n))^{1/p}B_{p}^n,
$$
we observe that
\begin{equation}\label{eq:asymp:pq1}
n^{-1/p}\lesssim \vol(B^n_{p,\infty})\lesssim \Bigl(\frac{1+\log(n)}{n}\Bigr)^{1/p}
\end{equation}
for all $0<p\le \infty.$ The aim of this section is to show, that the lower bound in \eqref{eq:asymp:pq1} is sharp
and that \eqref{eq:asymp:p} generalizes to all $0<p<\infty$ and $0<q\le\infty$ without additional logarithmic factors.

If $0<p<\infty$ and $q=1$, this can be obtained as a consequence of Theorem \ref{thm:q1}. Indeed, elementary estimates give
$$
\varkappa_p(k)\approx k^{1/p}
$$
with constants independent on $k$ and Theorem \ref{thm:q1} then implies
$$
\vol(B^n_{p,1})^{1/n}\approx \Bigl(\prod_{k=1}^n \varkappa_p(k)^{-1}\Bigr)^{1/n}\approx  \Bigl(\prod_{k=1}^n k^{-1/p}\Bigr)^{1/n}\approx (n!)^{-\frac{1}{p}\cdot\frac{1}{n}}.
$$
The result is then finished by another application of Stirling's formula.

To extend the result also to $q\not=1$, we apply the technique of entropy together with interpolation.

\begin{thm}\label{thm:asym:1}
Let $n\in\N$, $0<p<\infty$ and $0<q\le\infty$. Then
\begin{equation}\label{eq:thm:asymp}
\vol(B_{p,q}^n)^{1/n}\approx n^{-1/p}
\end{equation}
with the constants of equivalence independent of $n$.
\end{thm}
\begin{proof}
\emph{Step 1.:} First, we show the upper bound of $\vol(B^n_{p,\infty})^{1/n}$.
For that reason, we define the entropy numbers of a bounded linear operator between two quasi-normed Banach spaces $X$ and $Y$
as follows
$$
e_k(T:X\to Y)=\inf\Bigl\{\varepsilon>0:\exists \{y_l\}_{l=1}^{2^{k-1}}\subset Y\ \text{such that}\ T(B_X)\subset\bigcup_{l=1}^{2^{k-1}}(y_l+\varepsilon B_Y)\Bigr\}.
$$
Here, $B_X$ and $B_Y$ stand for the unit ball of $X$ and $Y$, respectively.

We use the interpolation inequality for entropy numbers (cf. Theorem \cite[Theorem 1.3.2 (i)]{ET})
together with the interpolation property of Lorentz spaces (cf. \cite[Theorems 5.2.1 and 5.3.1]{BL}) and obtain that
$$
e_{k}(id:\ell_{p,\infty}^n\to\ell_\infty^n)\le c_p e_k(id:\ell_{p/2}^n\to\ell_{\infty}^n)^{1/2},
$$
where $c_p>0$ depends only on $p$. Together with the known estimates of entropy numbers of embeddings of Lebesgue-type sequence spaces \cite{GL,KV,K2001,S},
we obtain
$$
e_{n}(id:\ell^n_{p,\infty}\to\ell_\infty^n)\le c_p n^{-1/p}.
$$
By the definition of entropy numbers, this means that $B_{p,\infty}^n$ can be covered with $2^{n-1}$ balls in $\ell_\infty^n$ with radius
$(1+\varepsilon)c_pn^{-1/p}$ for every $\varepsilon>0$. Comparing the volumes, we obtain
$$
\vol(B_{p,\infty}^n)\le 2^{n-1}[(1+\varepsilon)c_pn^{-1/p}]^n\vol(B_\infty^n),
$$
i.e. $\vol(B_{p,\infty}^n)^{1/n}\le c_p' n^{-1/p}.$

\emph{Step 2.:} The estimate from above for general $0<q\le \infty$ is covered by the embedding of Theorem \ref{thm:emb:1} and by the previous step.

\emph{Step 3.:} For the lower bound, we use again the interpolation
of entropy numbers leading to
$$
e_{n}(id:\ell_{p/2}^n\to \ell_{p,q}^n)\le c_{p,q} e_{n}(id:\ell_{p/2}^n\to \ell_{\infty}^n)^{1/2}.
$$
Therefore, the unit ball $B_{p/2}^n$ can be covered by $2^{n-1}$ copies of $B_{p,q}^n$ multiplied by $(1+\varepsilon)c n^{-1/p}$.
Comparing the volumes, we obtain
$$
c_1n^{-2/p}\le \vol(B_{p/2}^n)^{1/n}\le c_2n^{-1/p}\vol(B_{p,q}^n)^{1/n},
$$
which finishes the proof.
\end{proof}

The result of Theorem \ref{thm:asym:1} seems to be a bit surprising at first look - especially in view of \eqref{eq:asymp:pq1},
which suggests that some additional logarithmic factor could appear.
That the outcome of Theorem \ref{thm:asym:1} was by no means obvious is confirmed by inspecting the case $p=\infty$,
where the behavior of $n$-th root of the volume of the unit ball actually differs from \eqref{eq:thm:asymp}.

\begin{thm}\label{thm:asym:inf1}
Let $n\in \N$ be a positive integer. Then
\begin{equation}\label{eq:asym:inf1}
[\vol(B^n_{\infty,1})]^{1/n}\approx (\log (n+1))^{-1}
\end{equation}
with the constants of equivalence independent on $n$.
\end{thm}
\begin{proof}
By Theorem \ref{thm:q1}, we know that
$$
\vol(B^n_{\infty,1})^{1/n}\approx \Bigl(\prod_{k=1}^n \varkappa_\infty(k)^{-1}\Bigr)^{1/n},
$$
where
$$
\varkappa_\infty(k)=\sum_{j=1}^k \frac{1}{j}\approx \log(k+1)\quad \text{for any}\quad k\ge 1.
$$
Therefore,
$$
\vol(B^n_{\infty,1})^{1/n}\approx \Bigl(\prod_{k=1}^n \frac{1}{\log(k+1)}\Bigr)^{1/n}.
$$
The lower bound of this quantity is straightforward
$$
\Bigl(\prod_{k=1}^n \frac{1}{\log(k+1)}\Bigr)^{1/n}\ge
\Bigl(\prod_{k=1}^n \frac{1}{\log(n+1)}\Bigr)^{1/n}=\frac{1}{\log(n+1)}.
$$
For the upper bound, we use the inequality between geometric and arithmetic mean and obtain
$$
\Bigl(\prod_{k=1}^n \frac{1}{\log(k+1)}\Bigr)^{1/n}\le \frac{1}{n} \sum_{k=1}^n \frac{1}{\log(k+1)}\le \frac{1}{n}\biggl\{\frac{1}{\log(2)}+\int_2^{n+1}\frac{1}{\log(t)}dt\biggr\}.
$$
The last integral is known as the (offset) logarithmic integral and is known to be asymptotically $O(x/\log(x))$ for $x$ going to infinity, cf. \cite[Chapter 5]{AS}.
Alternatively, the same fact can be shown easily by the L'Hospital's rule.
This finishes the proof.
\end{proof}

\subsection{Ratio of volumes}

The unit balls $B^n_{p,\infty}$ of weak Lebesgue spaces
are commonly considered to be ``slightly larger'' than the unit balls of Lebesgue spaces
with the same summability parameter.
The aim of this section is to study their relation in more detail. For that sake, we define for $0<p<\infty$
\begin{equation}\label{eq:ratio:1}
R_{p,n}:=\frac{\vol(B_{p,\infty}^n)}{\vol(B_p^n)}.
\end{equation}
By the embedding in Theorem \ref{thm:emb:1} (which we give below with the full proof for reader's convenience, cf. \cite[Chapter 4, Proposition 4.2]{BS})
we know that this quantity is bounded from below by one.
Later on, we would like to study its behavior (i.e. growth) when $n$ tends to $\infty.$

\begin{thm}\label{thm:emb:1}
If $0<p<\infty$ and $0<q\le r\le \infty$, then
\begin{equation}\label{eq:emb:p2}
B_{p,q}^n\subset c_{p,q,r} B_{p,r}^n,
\end{equation}
where the quantity $c_{p,q,r}$ does not depend on $n$. In particular, $B_{p,q}^n\subset B_{p,r}^n$
if also $q\le p$.
\end{thm}
\begin{proof}
First, we prove the assertion with $r=\infty$. If $1\le l\le n$ is a positive integer, the result follows from
\begin{align*}
\|x\|_{p,q}^q=\sum_{k=1}^n k^{q/p-1}(x_k^*)^q\ge\sum_{k=1}^l k^{q/p-1}(x_k^*)^q \ge (x_l^*)^q\sum_{k=1}^lk^{q/p-1}
\end{align*}
and
$$
l^{q/p}(x_l^*)^q\le \|x\|_{p,q}^q\cdot l^{q/p}\cdot\Bigl(\sum_{k=1}^l k^{q/p-1}\Bigr)^{-1}.
$$
We obtain that
\begin{equation}\label{eq:emb:p1}
\|x\|_{p,\infty}=\max_{l=1,\dots,n}l^{1/p}x_l^*\le \|x\|_{p,q}\sup_{l\in\N}l^{1/p}\cdot\Bigl(\sum_{k=1}^l k^{q/p-1}\Bigr)^{-1/q}.
\end{equation}
For $q\le p$, it can be shown by elementary calculus that
$$
\sum_{k=1}^lk^{q/p-1}\ge l^{q/p}.
$$
Together with \eqref{eq:emb:p1}, this implies that $\|x\|_{p,\infty}\le \|x\|_{p,q}$
and we obtain $B_{p,q}^n\subset B_{p,\infty}^n,$ i.e. \eqref{eq:emb:p2} with $c_{p,q,\infty}=1.$
If, on the other hand, $q>p$ we estimate
$$
\sum_{k=1}^lk^{q/p-1}\ge\int_0^l t^{q/p-1}dt=\frac{p}{q}\cdot l^{q/p}
$$
and \eqref{eq:emb:p2} follows with $c_{p,q,\infty}=(q/p)^{1/q}$.

If $0<q< r<\infty$, we write
\begin{align*}
\|x\|_{p,r}&=\Bigl\{\sum_{k=1}^n k^{r/p-1}(x_k^*)^r\Bigr\}^{1/r}
=\Bigl\{\sum_{k=1}^n k^{q/p-1}(x_k^*)^q k^{(r-q)/p}(x_k^*)^{r-q}\Bigr\}^{1/r}\\
&\le \|x\|_{p,\infty}^{\frac{r-q}{r}}\cdot\|x\|_{p,q}^{q/r}\le [c_{p,q,\infty}\|x\|_{p,q}]^{\frac{r-q}{r}}\cdot\|x\|_{p,q}^{q/r},
\end{align*}
i.e.
$$
\|x\|_{p,r}\le c_{p,q,r}\|x\|_{p,q}\quad\text{with}\quad c_{p,q,r}=(c_{p,q,\infty})^{\frac{r-q}{r}}.
$$
\end{proof}

We show that the ratio $R_{p,n}$ defined in \eqref{eq:ratio:1} grows exponentially for $0<p\le 2$.
Naturally, we also conjecture that the same is true for all $0<p<\infty$, but we leave this as an open problem.

\begin{thm}\label{thm:ratio}
For every $0<p\le 2$, there is a constant $C_p>1$, such that
$$
R_{p,n}\gtrsim C_p^n
$$
with the multiplicative constant independent on $n$.
\end{thm}
\begin{proof}
We give the proof for even $n$'s, the proof for odd $n$'s is similar, only slightly more technical.
Let ${\mathcal B^n_p}\subset \R^n$ be the set of vectors $x\in\R^n$, which satisfy
$$
x_1^*\in\Bigl[\frac{1}{2^{1/p}},1\Bigr],\quad x_2^*\in \Bigl[\frac{1}{3^{1/p}},\frac{1}{2^{1/p}}\Bigr],\dots,x^*_{n/2}\in\Bigl[\frac{1}{(n/2+1)^{1/p}},\frac{1}{(n/2)^{1/p}}\Bigr]
$$
and $\displaystyle x^*_{n/2+1},\dots,x_n^*\in\Bigl[0,\frac{1}{n^{1/p}}\Bigr]$.
Then ${\mathcal B}_p^n\subset B_{p,\infty}^n$
and the volume of ${\mathcal B}_p^n$ can be calculated by combinatorial methods. Indeed,
there is ${n\choose n/2}$ ways how to choose the $n/2$ indices of the smallest coordinates. Furthermore, there is $(n/2)!$
ways, how to distribute the $n/2$ largest coordinates. We obtain that
\begin{align}\label{eq:ratio:proof1}
R_{p,n}&=\frac{\vol(B_{p,\infty}^{n})}{\vol(B_p^n)}\ge \frac{\vol({\mathcal B}_p^n)}{\vol(B_p^n)}\\
\notag&\ge\frac{\Gamma(1+n/p)}{\Gamma(1+1/p)^n}\cdot \binom{n}{n/2}\cdot (n/2)! \cdot\prod_{i=1}^{n/2}\frac{(i+1)^{1/p} - i^{1/p}}{i^{1/p}(i+1)^{1/p}} \cdot\left(\frac{1}{n^{1/p}}\right)^{n/2}.
\end{align}


First, we observe that, by Stirling's formula \eqref{eq:Stirling},
\begin{align}
\notag\Gamma(1+n/p)&\cdot\binom{n}{n/2}\cdot (n/2)! \cdot\prod_{i=1}^{n/2}\frac{1}{i^{1/p}(i+1)^{1/p}} \cdot\left(\frac{1}{n^{1/p}}\right)^{n/2}\\
\notag& = \frac{\Gamma(1+n/p)n!}{\left[(n/2)!\right]^{1+1/p}\left[(n/2+1)!\right]^{1/p}n^{n/(2p)}}\\
\label{eq:ratio:proof2}&\approx \frac{\sqrt{2\pi n/p}\left(\frac{n}{pe}\right)^{n/p}\sqrt{2\pi n}\left(\frac{n}{e}\right)^n}{\left(\sqrt{\pi n}\right)^{1+2/p}(n/2+1)^{1/p}\left(\frac{n}{2e}\right)^{n/2+n/p}n^{n/(2p)}}\\
\notag&\approx \left(\frac{2^{1/2+1/p}}{p^{1/p}e^{1/2}}\right)^n\cdot n^{n/2-n/(2p)+1/2-2/p}.
\end{align}

By the mean value theorem, we obtain
$$
(i+1)^{1/p} - i^{1/p}\geq \begin{cases}
\frac{i^{1/p-1}}{p},\quad 0<p\le 1,\\
\frac{(i+1)^{1/p-1}}{p},\quad 1<p\le 2.
\end{cases}
$$
We use \eqref{eq:Stirling} to estimate $\Gamma(1+1/p)$ together with \eqref{eq:ratio:proof1} and \eqref{eq:ratio:proof2} and obtain
\begin{align*}
R_{p,n} & \gtrsim \left(\frac{2^{1/2+1/p}}{\Gamma(1+1/p)p^{1/p}e^{1/2}}\right)^n\cdot \frac{n^{n/2-n/(2p)+1/2-2/p}}{p^{n/2}}
\cdot\left[(n/2)!\right]^{1/p-1}(n/2+1)^{\alpha(1/p-1)}\\
& \approx \left(\frac{2^{1+1/(2p)}}{\Gamma(1+1/p)p^{1/p+1/2}e^{1/(2p)}}\right)^n\cdot n^{-3/(2p)+\alpha(1/p-1)}
\gtrsim \left(\frac{2^{1/2+1/(2p)}}{\pi^{1/2}e^{p/12-1/(2p)}}\right)^n n^{-3/(2p)+\alpha(1/p-1)},
\end{align*}
where $\alpha=0$ for $0<p\le 1$ and $\alpha=1$ for $1<p\le 2.$
The proof is then finished by monotonicity and
$$
\frac{2^{1/2+1/(2p)}}{\pi^{1/2}e^{p/12-1/(2p)}}=\sqrt{\frac{2}{\pi}} \cdot\frac{ (2e)^{1/(2p)}}{e^{p/12}}\ge
\sqrt{\frac{2}{\pi}}\cdot \frac{(2e)^{1/4}}{e^{1/6}}>1.
$$

\end{proof}

\section{Entropy numbers}

We have already seen the closed connection between estimates of volumes of unit balls of 
finitedimensional (quasi-)Banach spaces and the decay of entropy numbers of embeddings of such spaces
in the proof of Theorem \ref{thm:asym:1}.
With the same arguments as there, it is rather straightforward to prove that
\begin{equation}\label{eq:interpol:1}
e_k(id:\ell^n_{p_0,q_0}\to \ell_{p_1,q_1}^n)\approx e_k(id:\ell^n_{p_0}\to \ell_{p_1}^n)
\end{equation}
for $0<p_0,p_1<\infty$ with $p_0\not=p_1.$
On the other hand, it was shown in \cite{EN}, that the entropy numbers of diagonal operators between Lorentz sequence spaces
can exhibit also a very complicated behavior. Actually, they served in \cite{EN} as the first counterexample to a commonly conjectured interpolation
inequality for entropy numbers.

We complement \eqref{eq:interpol:1} by considering  the limiting case $p_0=p_1.$
As an application of our volume estimates, accompanied by further arguments, we will investigate in this section
the decay of the entropy numbers $e_k(id:\ell^n_{1,\infty}\to \ell^n_{1}).$

Before we come to our main result, we state a result from coding theory \cite{coding:1,coding:2}, which turned out to be useful
also in connection with entropy numbers \cite{EN,KV} and even in optimality of sparse recovery in compressed sensing \cite{BCKV,FPRU,FR}.

\begin{lem}\label{Lemma:coding} Let $k\le n$ be positive integers. Then there are $M$ subsets $T_1,\dots,T_M$ of $\{1,\dots,n\}$,
such that
\begin{enumerate}
\item[(i)] $\displaystyle M\ge \Bigl(\frac{n}{4k}\Bigr)^{k/2}$,
\item[(ii)] $|T_i|=k$ for all $k=1,\dots,M$,
\item[(iii)] $|T_i\cap T_j|<k/2$ for all $i\not=j.$
\end{enumerate}
\end{lem}

To keep the argument simple, we restrict ourselves to $p=1$.
\begin{thm}\label{thm:entropy} Let $k$ and $n$ be positive integers. Then
$$
e_k(id:\ell_{1,\infty}^n\to \ell_1^n)\approx\begin{cases} \log(1+n/k),\quad 1\le k\le n,\\
2^{-\frac{k-1}{n}},\quad k\ge n,
\end{cases}
$$
where the constants of equivalence do not depend on $k$ and $n$.
\end{thm}
\begin{proof}
\emph{Step 1. (lower bound for $k\ge n$):} If $B_{1,\infty}^n$ is covered by $2^{k-1}$ balls in $\ell_1^n$ with radius $\varepsilon>0$, it must hold
$$
\vol(B_{1,\infty}^n)^{1/n}\le 2^{\frac{k-1}{n}}\varepsilon\vol(B_1^n)^{1/n},
$$
which (in combination with Theorem \ref{thm:asym:1}) gives the lower bound for $k\ge n$.\\
\emph{Step 2. (upper bound for $k\ge n$):} We use again volume arguments.
Let $\varepsilon>0$ be a parameter to be chosen later on.
Let $\{x_1,\dots,x_N\}\subset B_{1,\infty}^n$
be a maximal $\varepsilon$-distant set in the metric of $\ell_1^n.$ This means that
$$
B_{1,\infty}^n\subset \bigcup_{j=1}^N (x_j+\varepsilon B_1^n)
$$
and $\|x_i-x_j\|_1>\varepsilon$ for $i\not =j.$ Hence, any time $N\le 2^{k-1}$ for some positive integer $k$, then
$e_k(id:\ell_{1,\infty}^n\to \ell_1^n)\le \varepsilon.$ To estimate $N$ from above, let us note that $(x_j+\varepsilon B_{1}^n)\subset 2(1+\varepsilon)B_{1,\infty}^n$,
which follows by the quasi-triangle inequality for $\ell_{1,\infty}^n$. On the other hand, the triangle inequality of $\ell_1^n$ implies that
$(x_j+\frac{\varepsilon}{2}B_1^n)$ are mutually disjoint. Hence,
$$
N\Bigl(\frac{\varepsilon}{2}\Bigr)^n\vol(B_1^n)\le 2^n(1+\varepsilon)^n\vol(B_{1,\infty}^n),
$$
i.e.
\begin{equation}\label{eq:entropy:1}
N\le 4^n\Bigl(1+\frac{1}{\varepsilon}\Bigr)^n\frac{\vol(B_{1,\infty}^n)}{\vol(B_{1}^n)}\le 
\frac{8^n}{\varepsilon^n}\frac{\vol(B_{1,\infty}^n)}{\vol(B_{1}^n)}
\end{equation}
if $0<\varepsilon<1.$ We now define the parameter $\varepsilon$ by setting the right-hand side of \eqref{eq:entropy:1} equal to $2^{k-1}$.
By Theorem \ref{thm:asym:1}, there exists an integer $\gamma\ge 1$, such that
$\varepsilon<1$ for $k\ge \gamma n$. In this way, we get $N\le 2^{k-1}$
and $\varepsilon=8 [\vol(B_{1,\infty}^n)/\vol(B_{1}^n)]^{1/n}\cdot 2^{-\frac{k-1}{n}}\le c\,2^{-\frac{k-1}{n}}$. This gives the result
for $k\ge \gamma n.$

\emph{Step 3. (upper bound for $k\le n$):} Let $1\le l\le n/2$ be a positive integer, which we will chose later on.
To every $x\in B_{1,\infty}^n$, we associate $S\subset\{1,\dots,n\}$ to be the indices of its $l$ largest entries (in absolute value).
Furthermore, $x_S\in\R^n$ denotes the restriction of $x$ to $S$. 
We know that
$$
\|x-x_S\|_1\le \sum_{k=l+1}^n \frac{1}{k}\le \int_{l}^{n}\frac{dx}{x}=\log(n)-\log(l)=\log(n/l).
$$
By Step 2, there is an absolute constant $c>0$ (independent of $l$), such that
$$
e_{\gamma l}(id:\ell_{1,\infty}^l\to \ell_1^l)< c,
$$
where $\gamma\ge 1$ is the integer constant from Step 2.

Hence, there is a point set ${\mathcal N}\subset \R^l$, with $|{\mathcal N}|=2^{\gamma l}$,
which is a $c$-net of $B_{1,\infty}^l$ in the $\ell_1^l$-norm. For any set $S$ as above, we embed ${\mathcal N}$ into $\R^n$ by extending the points from
${\mathcal N}$ by zero outside of $S$ and obtain a point set ${\mathcal N}_S$, which is a $c$-net of 
$\{x\in B_{1,\infty}^n:\supp(x)\subset S\}$. Taking the union of all these nets over all sets $S\subset \{1,\dots,n\}$ with $|S|=l$,
we get $2^{\gamma l}{n\choose l}$ points,
which can approximate any $x\in B_{1,\infty}^n$ within $c+\log(n/l)$
in the $\ell_1^n$-norm.

We use the elementary estimate ${n\choose l}\le (en/l)^l$ and assume (without loss of generality) that $\gamma\ge 2$.
Then we may conclude, that whenever $2^{k-1}\ge 2^{\gamma l}{n\choose l}$, we have $e_k(id:\ell_{1,\infty}^n\to\ell_{1}^n)\le c+\log(n/l)$, i.e.
$$
k-1\ge \gamma l(1+\log(en/l))\implies e_k(id:\ell_{1,\infty}^n\to\ell_{1}^n)\lesssim (1+\log(n/l)).
$$
By a standard technical argument, $l$ can be chosen up to the order of $k/\log(n/k)$, which gives the result
for $n$ large enough and $k$ between $(\gamma+1)\log(en)$ and $n$. Using monotonicity of entropy numbers, the upper bound from Step 2
and the elementary bound $e_k(id:\ell_{1,\infty}^n\to\ell_{1}^n)\le \|id:\ell_{1,\infty}^n\to\ell_{1}^n\|\le 1+\log(n)$ concludes the proof of the upper bounds.

\emph{Step 4. (lower bound for $k\le n$):} Let $n$ be a sufficiently large positive integer and let $\nu\ge 1$ be the largest integer with $12\cdot 4^\nu\le n$.
Let $1\le\mu\le \nu$ be a positive integer.

We apply Lemma \ref{Lemma:coding} with $k$ replaced by $4^l$ for every integer $l$ with $\mu\le l \le \nu$.
In this way, we obtain a system of subsets $T^l_1,\dots,T^l_{M_l}$ of $\{1,\dots,n\}$, such that
$|T^l_i|=4^l$ for every $1\le i \le M_l$, $|T^l_i\cap T^l_j|<4^l/2$ for $i\not=j$
and
$$
\displaystyle M_l\ge \Bigl(\frac{n}{4^{l+1}}\Bigr)^{4^l/2}\ge M:=\Bigl(\frac{n}{4^{\mu+1}}\Bigr)^{4^\mu/2}.
$$

For $j\in\{1,\dots,M\}$, we put
\begin{align*}
\widetilde T^\mu_j&=T^\mu_j,\\
\widetilde T^{\mu+1}_j&=T^{\mu+1}_j\setminus T^{\mu}_j,\\
&\vdots\\
\widetilde T^{\nu}_j&=T^{\nu}_j\setminus (T^{\nu-1}_j\cup\dots\cup T^{\mu}_j).
\end{align*}
Observe, that by this construction the sets $\{\widetilde T^{l}_j:\mu\le l\le \nu\}$ are mutually disjoint and $|\widetilde T^l_j|\le |T^l_j|=4^l$.
Furthermore, $|\widetilde T^\mu_j|=4^\mu$ and
\begin{align}
\notag |\widetilde T^l_j|&\ge |T^{l}_j|- [|T^{l-1}_j|+\dots+|T^{\mu}_j|]=4^l-[4^{l-1}+\dots+4^\mu]\\
\label{eq:entro:11}&\ge 4^l\Bigl(1-\sum_{s=1}^\infty\frac{1}{4^s}\Bigr)=\frac{2}{3}\cdot 4^l
\end{align}
for $\mu<l\le\nu.$

We associate to the sets $\{\widetilde T_j^l:\mu\le l\le \nu, 1\le j\le M\}$ a system of vectors $x^1,\dots,x^M\in{\mathbb R}^n$.
First, we observe that if $u\in\{1,\dots,n\}$ belongs to $\widetilde T^l_j$ for some $l\in\{\mu,\mu+1,\dots,\nu\}$, then this $l$ is unique and we put $(x^j)_u=\frac{1}{4^l}.$
Otherwise, we set $(x^j)_u=0.$ We may also express this construction by
$$
x^j=\sum_{l=\mu}^{\nu}\frac{1}{4^l}\chi_{\widetilde T_j^l},
$$
where $\chi_A$ is the indicator function of a set $A$.

Now we observe that
\begin{align*}
\|x^j\|_{1,\infty}&\le \max\Bigl\{4^\mu\cdot\frac{1}{4^\mu}, \frac{4^{\mu}+4^{\mu+1}}{4^{\mu+1}},
\dots,\frac{4^{\mu}+4^{\mu+1}+\dots+4^\nu}{4^{\nu}}\Bigr\}\\
&\le 1+\frac{1}{4}+\frac{1}{4^2}+\dots=\frac{4}{3}.
\end{align*}
Furthermore, let $i\not=j$ and  let $u\in \widetilde T^l_i$ with $u\not\in\widetilde T^l_j$. Then
\begin{equation}\label{eq:entro:12}
|(x^i)_u-(x^j)_u|\ge \frac{1}{4^l}-\frac{1}{4^{l+1}}=\frac{3}{4}\cdot\frac{1}{4^l}.
\end{equation}
To estimate the $\ell_1$-distances among the points $\{x^1,\dots,x^M\}$, we combine \eqref{eq:entro:12}, \eqref{eq:entro:11}, 
and obtain for $i\not= j$
\begin{align*}
\|x^i-x^j\|_1&\ge\sum_{l=\mu}^\nu \sum_{u\in\widetilde T_i^l\setminus \widetilde T^l_j}|(x^i)_u-(x^j)_u|
\ge \sum_{l=\mu}^\nu |\widetilde T^l_i\setminus \widetilde T^l_j|\cdot \frac{3}{4}\cdot\frac{1}{4^l}\\
&=\frac{3}{4}\Bigl\{\sum_{l=\mu}^\nu |\widetilde T^l_i|\cdot \frac{1}{4^l}-\sum_{l=\mu}^\nu |\widetilde T^l_i \cap \widetilde T^l_j|\cdot \frac{1}{4^l}\Bigr\}\\
&\ge \frac{3}{4}\Bigl\{1+\sum_{l=\mu+1}^\nu \frac{2}{3}\cdot 4^l\cdot\frac{1}{4^l}
-\sum_{l=\mu}^\nu |T^l_i \cap T^l_j|\cdot \frac{1}{4^l}\Bigr\}\\
&\ge \frac{3}{4}\Bigl\{1+\frac{2}{3}(\nu-\mu)
-\sum_{l=\mu}^\nu \frac{4^l}{2}\cdot \frac{1}{4^l}\Bigr\}\\
&= \frac{3}{4}\Bigl\{1+\frac{2}{3}(\nu-\mu)-\frac{1}{2}(\nu-\mu+1)\Bigr\}\ge \frac{1}{8}(\nu-\mu+1).
\end{align*}
We conclude, that the points $\{x^j:j=1,\dots,M\}$ satisfy
$$
\|x^j\|_{1,\infty}\le \frac{4}{3}\quad \text{and}\quad  \|x^i-x^j\|_1\ge \frac{1}{8}(\nu-\mu+1)\ \text{for}\ i\not=j.
$$
It follows that if a positive integer $k$ satisfies
\begin{equation}\label{eq:entropy:2}
M=\Bigl(\frac{n}{4^{\mu+1}}\Bigr)^{4^\mu/2}\ge 2^{k-1},
\end{equation}
then
$$
e_k(id:\ell_{1,\infty}^n\to\ell_1^n)\ge c(\nu-\mu+1),
$$
where the absolute constant $c$ can be taken $c=\frac{3}{64}.$

Let now $n\ge 200$ and $1\le k\le n/200$ be positive integers. Then we chose $\nu\ge 1$ to be the largest integer with $12\cdot 4^\nu\le n$
and let $\mu\ge 1$ be the smallest integer with $k\le 4^\mu/2.$
Due to $\frac{n}{4^{\mu+1}}\ge 2$, this choice ensures \eqref{eq:entropy:2} and
$$
\nu-\mu+1\ge \log_4\Bigl(\frac{n}{48}\Bigr)-\log_4(2k)\gtrsim \log(1+n/k).
$$
The remaining pairs of $k$ and $n$ are covered by monotonicity of entropy numbers at the cost of constants of equivalence.
\end{proof}

{\bf Acknowledgement:} We would like to thank Franck Barthe for proposing the problem to us
and to Leonardo Colzani and Henning Kempka for valuable discussions.

\end{document}